\documentclass[12pt,a4paper]{article}
\usepackage{amsmath, amsfonts, amsthm,amssymb}
\usepackage{graphicx}
\usepackage{float}
\usepackage{plain}
\usepackage[usenames]{color}
\usepackage[usenames,dvipsnames]{xcolor}
\usepackage{verbatim}
\usepackage{ wasysym }

\numberwithin{equation}{section}

\newtheorem{theorem}{Theorem}[section]
\newtheorem{lemma}[theorem]{Lemma}

\newtheorem{proposition}[theorem]{Proposition}

\newtheorem{definition}[theorem]{Definition}
{
    \catcode`\@=11
    \gdef\dima{\mathop{\operator@font dim_{A}}\nolimits}
    \gdef\dimh{\mathop{\operator@font dim_{ H}}\nolimits}
    \gdef\dimb{\mathop{\operator@font dim_{ B}}\nolimits}
    \gdef\dims{\mathop{\operator@font dim_{ S}}\nolimits}
    \gdef\diam{\mathop{\operator@font diam}\nolimits}
    \gdef\span{\mathop{\operator@font span}\nolimits}
    \gdef\rank{\mathop{\operator@font rank}\nolimits}
    \gdef\dh{\mathop{\operator@font d_{\cal H}}\nolimits}
    \gdef\ph{\mathop{\operator@font p_{\cal H}}\nolimits}
    \gdef\cover{\mathop{\operator@font {\cal N}}\nolimits}
    \gdef\cone{\mathop{\operator@font {\cal C}}\nolimits}
    \gdef\card{\mathop{\operator@font card}\nolimits}
}
\def\R{\mathbb{R}}
\def\N{\mathbb{N}}
\def\I{\mathcal{I}}
\def\J{\mathcal{J}}

\def\words#1{\quad\hbox{#1}\quad}

\title{On the Assouad dimension of self-similar sets with overlaps}

\author{J. M. Fraser$^{1}$,
	A. M. Henderson$^{2}$,
	E. J. Olson$^{3}$,
	J. C. Robinson$^{4}$}

\begin{document}
      \maketitle

\begin{center}
$^1$School of Mathematics, The University of Manchester, Manchester, M13 9PL, UK.\\ \vspace{3mm} 
$^2$Department of Mathematics, University of California, Riverside, 900 University Ave., Riverside, CA 92521, USA.\\ \vspace{3mm} 
$^3$Department of Mathematics/084, University of Nevada, Reno, NV 89557, USA.\\ \vspace{3mm} 
$^4$Mathematics Institute, Zeeman Building, University of Warwick, Coventry, \\CV4 7AL, UK.
\end{center}
\vspace{0mm}
\begin{abstract}
  It is known that, unlike the Hausdorff dimension, the Assouad dimension
  of a self-similar set can exceed the similarity dimension if there are
  overlaps in the construction.  Our main result is the following precise
  dichotomy for self-similar sets in the line: either the \emph{weak
  separation property} is satisfied, in which case the Hausdorff and
  Assouad dimensions coincide; or the \emph{weak separation property}
  is not satisfied, in which case the Assouad dimension is maximal
(equal to one).
  In the first case we prove that the self-similar set
  is Ahlfors regular, and in the second case we use the fact that if the
  \emph{weak separation property} is not satisfied, one can approximate
  the identity arbitrarily well in the group generated by the similarity
  mappings, and this allows us to build a \emph{weak tangent} that
  contains an interval.  We also obtain results in higher dimensions
  and provide illustrative examples showing that the `equality/maximal'
  dichotomy does not extend to this setting. \\

\noindent \emph{AMS Classification 2010:} 28A80, 37C45 (primary);
	28A78 (secondary). \\


\end{abstract}

\section{Introduction}

Self-similar sets are arguably the most important examples of fractal
sets.  This is due to their simple description and the fact that they
exhibit many of the properties one expects from fractals.  Although
they have been studied extensively in the literature over the past
half century, there are still many fascinating and challenging open
problems available, particularly relating to the overlapping case, that is, self-similar sets $F$ such that every iterated function
system with attractor $F$ fails to satisfy the open set condition.
Our paper focuses on the Assouad dimension of such sets.

Given a bounded set $E \subseteq \R^d$ where $d\in\N$ and $\rho>0$,
let $\cover_E(\rho)$ be the smallest number of open balls of radius
$\rho$ required to cover $E$.  Given a (potentially unbounded) set $X
\subseteq \R^d$, let
$$
\cover_X(r,\rho) \ =  \  \max\big\{\, \cover_{X\cap B_r(x)}(\rho) :
x\in X\, \big\},
$$
where
$$
B_r(x)=\{\, y\in\R^d : \|x-y\|<r\,\}
$$
denotes the open ball of radius $r$ centered at $x$.  The {\it Assouad
  dimension\/} $\dima X$ is then defined to be the infimum over all
$s$ for which there exists $K_s$ such that $\cover_X(r,\rho)\le K_s
(r/\rho)^s$ for all $0<\rho<r\le 1.$ We note that some authors allow
$r$ (and $\rho<r$) to be arbitrarily large, but for bounded sets $X$
this does not change the dimension.

This dimension was introduced by Bouligand \cite {Bouligand1928} to
study bi-Lipschitz embeddings and was further studied by Assouad in
\cite{Assouad1977, Assouad1979}.  Although an interesting notion in
its own right, it has found most prominence in the literature due to
its relationship with embeddability problems and quasi-conformal
mappings.  For example, Olson \cite{Olson2002} proved that for $X
\subseteq \R^d$, $\dima (X-X)<s$ implies almost every orthogonal
projection of rank $s$ is injective on $X$ and results in an embedding
that is bi-Lipschitz except for a logarithmic correction.  This result
was extended to subsets of infinite-dimensional Hilbert spaces by
Olson and Robinson in \cite{Olson2010}. Further work relating to
Assouad dimension and embeddings may be found in
Heinonen~\cite{Heinonen2001}, Luukkainen~\cite{Luukkainen1998}, Mackay
and Tyson~\cite{Mackay2010} and Robinson~\cite{Robinson2011}.  Concerning quasi-conformal mappings, there is widespread interest in computing the conformal (Hausdorff or Assouad) dimension in various settings.  Here the conformal dimension of $X$ is the infimum of the dimension of quasi-conformal images of $X$.   Keith and Laakso~\cite{KeithLaakso04} study Ahlfors regular metric spaces whose Assouad dimension cannot be lowered by quasi-conformal maps and Bonk and Kleiner~\cite{BonkKleiner05} apply this work to resolve Cannon's conjecture in certain settings.  Tyson~\cite{Tyson01} proved that if the conformal Assouad dimension of a bounded set in Euclidean space is strictly less than 1, then it is 0.
Recently, the Assouad dimension has also been gaining substantial
attention in the literature on fractal geometry.

Let $\I = \{1, \dots, \lvert \I \vert \}$ be a finite index set and $\{ S_i \}_{i \in \I}$ be a set of
contracting self-maps on some compact subset of $\R^d$.  Such a
collection of maps is called an {\it iterated function system} (IFS).
It is well known, see for example
Edgar~\cite{Edgar2008} or Falconer~\cite{Falconer1985}, that there exists a unique non-empty compact
set $F \subseteq \R^d$ satisfying
\begin{equation}\label{attractor}
  F = \bigcup_{i\in \I} S_i (F)
\end{equation}
which is called the \emph{attractor} of the iterated function system.
In this paper we shall work exclusively in the setting where the maps
are \emph{similarities}. A map $S\colon \R^d\to \R^d$ is called a {\it
  similarity\/} if there exists a ratio $c>0$ such that
$$
\|S(x)-S(y)\|=c \|x-y\| \words{for all} x,y\in \R^d.
$$
If $c\in (0,1)$ then $S$ is called a {\it contracting similarity}.  A
{\it self-similar set\/} is any set $F\subset\R^d$ for which there
exists an iterated function system of contracting similarities
$\{S_i\}_{i\in\I}$ that has $F$ as its attractor.  Without loss of
generality we will assume from now on that $F \subseteq [0,1]^d$.  We say $F$ is \emph{non-trivial} if it is not a singleton.  The trivial situation can only occur if all the defining maps have a common fixed point, in which case the attractor is just this point.

Given an iterated function system of contracting similarities, the
\emph{similarity dimension} $\dims (\{S_i\}_{i\in\I})$ is defined to
be the maximum of the spatial dimension $d$ and the unique solution
$s$ of the \emph{Hutchinson--Moran formula}
\[
\sum_{i \in \I} c_i^s = 1.
\]
Note that the similarity dimension is always an upper bound for the
Hausdorff dimension of $F$ which we shall denote by $\dimh F$.  The
Assouad dimension is also an upper bound on the Hausdorff dimension
(see \cite{Robinson2011}, for example).

Most of the commonly used notions of dimension coincide for arbitrary
self-similar sets, regardless of separation conditions.  Indeed,
equality of the box, packing and Hausdorff dimensions has been long
known and follows, for example, from the implicit theorems of
Falconer~\cite{Falconer1989} and McLaughlin~\cite{McLaughlin1987}.
The additional equality with the lower dimension, which is the natural
dual to the Assouad dimension and is \emph{a priori } just a lower
bound for the Hausdorff dimension, was proved by Fraser \cite[Theorem
2.11]{Fraser2013}.  For this reason we will omit discussion of all
notions of dimension apart from the Hausdorff and Assouad dimensions.
For a review of the other notions see Falconer \cite[Chapters
2-3]{Falconer2003} or Robinson \cite{Robinson2011}.

In \cite{Moran1946} Moran introduced the \emph{open set condition}
(OSC) which, if satisfied, guarantees equality of the Hausdorff
dimension and the similarity dimension, see for example ~\cite[Section
9.3]{Falconer2003}.  It can occur that the Hausdorff dimension is
strictly less than the similarity dimension if different iterates of
maps in the iterated function system overlap exactly.  An example of
this dimension drop is given by the modified Sierpi\'nski triangle of
Ngai and Wang \cite[Figure 2, page 670]{Ngai2001}.  It is an important
open problem to decide if exact overlaps are the only way such a
dimension drop can occur, see for example Peres and
Solomyak~\cite[Question 2.6]{Peres2000}.  Recently an important step
towards solving this problem has been made by Hochman
\cite{Hochman2013}, where it is shown in $\R$, assuming the defining
parameters for the iterated function system are algebraic, that the
only cause for a dimension drop is exact overlaps.  Hochman's
arguments rely on a careful application of ergodic theory and in
establishing inverse theorems for the growth of certain entropy
functions.

\begin{definition}
  An IFS $\{S_i\}_{i \in \I}$ satisfies the
  \emph{open set condition} if there exists a non-empty open set $U$
  such that
  \[
  \bigcup_{i \in \I} S_{i}(U) \subseteq U
  \]
  with the union disjoint.
\end{definition}

Note that the open set $U$ need not be unique.  Falconer \cite[page
122]{Falconer1985} shows that if the OSC is satisfied for some $U$, then the attractor $F\subseteq \overline U$ and a result of
Schief \cite{Schief1994} is that we may assume that $F \cap U \neq
\emptyset$.  Notably, if the OSC is satisfied, then the Assouad
dimension also coincides with the Hausdorff, and thus similarity,
dimension.  This result, which has been known in the folklore and
essentially dates back to
Hutchinson, is due to the fact that self-similar sets are Ahlfors regular,
which is a sufficient condition for equality of Hausdorff and Assouad
dimension. With $\mathcal{H}^s$ denoting the $s$-dimensional
Hausdorff measure, recall that a set $F \subseteq \R^d$ is called
\emph{Ahlfors regular} if there exist constants $a,b>0$ such that
$$
a\, r^s \leq \mathcal{H}^s \big( B_r(x) \cap F \big) \leq b \, r^s
$$
for $s=\dimh F$, all $x \in F$ and all $0<r< 1$, see Heinonen \cite[Chapter
8]{Heinonen2001}.  We note that it is sufficient to show that there exists some $r_0>0$ such that these bounds hold for all $r \in (0, r_0)$. Without relying on the sophisticated notion of
Ahlfors regularity, simple direct proofs are also possible, see for
example Olson \cite{Olson2002}, Olsen \cite{Olsen2011}, Henderson
\cite{Henderson2013}, and Fraser \cite{Fraser2013}.  Olsen
\cite[Question 1.3]{Olsen2011} asked the natural question of whether
$\dimh F=\dima F$ holds for any self-similar set $F$, regardless of
separation properties.  This was answered in the negative by Fraser
\cite[Section 3.1]{Fraser2013}.  Similar techniques were used by
Henderson \cite{Henderson2013} to find examples of self-similar sets
such that $\dima F$ is arbitrarily small but where $\dima(F-F)=1$.
The Assouad dimension of more general attractors has also been
considered by, for example, Mackay \cite{Mackay2011} and Fraser
\cite{Fraser2013} who studied certain planar self-affine
constructions.

The weak separation property of Lau and Ngai \cite{Lau1999} and
Zerner~\cite{Zerner1996} will be our primary tool for studying
self-similar sets with overlaps and, indeed, provides a pleasing
dichotomy concerning Assouad dimension. Before describing the weak
separation property we introduce some notation.

Let $\I^* = \bigcup_{k\geq1} \I^k$ be the set of all finite sequences
with entries in $\I$.  For
\[
\alpha= \big(i_1, i_2, \ldots, i_k \big) \in \I^*
\]
write
\[
S_{\alpha} = S_{i_1} \circ S_{i_2} \circ \cdots \circ S_{i_k}
\]
and
\[
c_{\alpha} = c_{i_1} c_{i_2} \cdots c_{i_k}.
\]
Note that $S_\alpha$ is a contracting similarity with ratio
$c_\alpha$.  Let
\[
\overline{\alpha} = (i_1,i_2, \ldots, i_{k-1}) \in \I^* \cup \{
\omega\},
\]
where $\omega$ is the empty sequence of length zero.  For notational
convenience define $S_\omega = I$ to be the identity map and
$c_\omega=1$.  Let
\[
\mathcal{E} = \{ S_{\alpha}^{-1} \circ S_{\beta} : \alpha, \beta \in
\I^* \text{ with } \alpha \neq \beta \}
\]
and equip the group of all similarities on $\R^d$ with the topology
induced by pointwise convergence.  We now define the weak separation
property.
\begin{definition}[Lau and Ngai \cite{Lau1999}, Zerner \cite{Zerner1996}]\label{WSP}
  An IFS $\{S_i\}_{i \in \I}$ satisfies the
  \emph{weak separation property} if
  \[
  I \notin \overline{ \mathcal{E} \setminus \{I\}}.
  \]
\end{definition}
The weak separation property says that we cannot find a sequence of
pairs of maps $( S_{\alpha}, S_{\beta} )$ such that
$S_\alpha^{-1}\circ S_\beta$ gets arbitrarily close but not equal to
the identity.  Intuitively, this means that the images $S_\alpha(F)$ of the attractor $F$ 
may only overlap in a limited number of ways.

It was shown by Zerner \cite{Zerner1996} that the open set condition is
strictly stronger than the weak separation property.  In particular
using results from Schief \cite{Schief1994} it was shown that the open set
condition is equivalent to $ I \notin \overline{ \mathcal{E} }$. Since
all maps in $\mathcal{E}$ are similarities, taking the pointwise
closure of $\mathcal{E}$ is equivalent to taking the closure in the
space of similarities $S\colon [0,1]^d \to \R^d$ equipped with the
uniform norm $ \| \cdot \|_{L^\infty([0,1]^d)}$ which we denote by $\|
\cdot \|_\infty$ throughout, and where the value of $d$ will be clear from
the context.

Our main results may be stated as follows. In the case of self-similar
sets in the real line we obtain a precise dichotomy.

\begin{theorem}\label{mainR}
  Let $F$ be a non-trivial self-similar set in $\R$.  If the defining IFS for $F$ satisfies
  the weak separation property, then $\dima F = \dimh F$; otherwise,
  $\dima F=1$.
\end{theorem}

In the case of self-similar sets in $\R^d$ we obtain a less precise
dichotomy.

\begin{theorem}\label{mainRd}
  Let $F$ be a self-similar set in $\R^d$ that is not
  contained in any $(d-1)$-dimensional hyperplane.
  If the defining IFS for $F$ satisfies the weak separation
  property, then $\dima F = \dimh F$; otherwise, $\dima F \ge 1$.
\end{theorem}

It is straightforward to see that our assumption that $F$ is not contained in a hyperplane is required.  For example, let $F$ be the middle third Cantor set, embedded in the natural way in $\mathbb{R}^3$, and view it as the attractor of an IFS consisting of three mappings, the first two being the expected maps (the natural extensions of the maps in the one dimensional IFS to $\mathbb{R}^3$) and the third one being one of the first two maps composed with an irrational rotation around the one dimensional subspace containing $F$.  Indeed, the WSP is not satisfied for this system, but $\dima F < 1$.  The problem is that the WSP \emph{is} satisfied when one restricts the defining maps to the one dimensional subspace containing $F$.  It is also easy to see that this condition is not restrictive at all because
one can always restrict to a linear space with the appropriate
dimension, i.e.\ where one cannot further restrict to any hyperplane,
and then apply Theorem \ref{mainRd}.


Theorems \ref{mainR} and
\ref{mainRd} may be broken into two parts.  The first case covers what
happens when the weak separation property holds and the second case
covers when the weak separation property fails.  Section \ref{WSPtrue}
treats the first case; Section \ref{WSPfalse} treats second case.
After proving our main results we explore some examples which
illustrate our theorems and in particular show that the `equality/maximal' dichotomy seen in Theorem \ref{mainR} does not extend to $\mathbb{R}^d$, i.e., there exist self-similar sets in $\mathbb{R}^d$ (not contained in any hyperplane) with Assouad dimension strictly between the Hausdorff dimension and the ambient spatial dimension. We close
with some remarks concerning how we might improve our
understanding of the Assouad dimension of a self-similar set
in $\R^d$ when the weak
separation property is not satisfied.

\section{Systems with the Weak Separation Property}\label{WSPtrue}

Our first result shows that the weak separation property implies that the
Hausdorff and Assouad dimensions coincide.

\begin{theorem} \label{thmWSP} Let $F$ be a self-similar set in $\R^d$
  not contained in any $(d-1)$-dimensional hyperplane.  If the defining IFS for $F$ satisfies the weak
  separation property, then $F$ is Ahlfors regular and,
  in particular, $ \dimh F = \dima F$.
\end{theorem}
 Before beginning our proof we introduce the notations
  \begin{equation}\label{Iris}
  \I_r = \big\{ \alpha \in \I^* : c_\alpha \leq r <
  c_{\overline{\alpha}} \big\}
  \end{equation}
  and
  \[
  \I_r(x) = \big\{ \alpha \in \I_r : B_r(x) \cap S_\alpha(F) \neq
  \emptyset \big\}.
  \]
Note that for $r \in (0,1)$ we have
$$
	F=\bigcup_{\alpha \in\I_r} S_\alpha(F)
\words{and}
	B_r(x)\cap F\subseteq
	\bigcup_{\alpha \in\I_r(x)} S_\alpha(F).
$$

We remark that the common value for the Hausdorff and Assouad
dimensions given by Theorem \ref{thmWSP} may be strictly less than the
similarity dimension. This happens when exact overlaps allow certain
compositions of maps to be deleted.  In fact, if one takes the infimum
over all such `reduced IFSs', then Zerner \cite[Theorem 2]{Zerner1996}
shows that this gives the common value.
In particular, if the WSP is
satisfied and $F$ is not contained in a hyperplane, then
\[
\dimh F= \dima F=\lim_{ r \to 0} \, \dims(\{S_\alpha : \alpha
  \in \I_r\}) = \inf_{ r \in (0,1)} \, \dims(\{S_\alpha : \alpha
  \in \I_r\}) .
\]
To see why this number drops from the original similarity dimension if
there are exact overlaps, observe that $\{S_\alpha : \alpha
  \in \I_r\}$ is a \emph{set} of maps
rather than a \emph{multiset}, i.e.\ repeated maps are included only
once.  Zerner \cite[Proposition 2]{Zerner1996} also shows that the
number on the right of the above equation is equal to the original
similarity dimension if and only if the OSC is satisfied.  The problem
is that this value may be very difficult to compute.  Ngai and Wang
\cite{Ngai2001} introduced the \emph{finite type condition} which is
strictly stronger than the WSP (see Nguyen \cite{Nguyen2002}),
but allows the above value to be computed explicitly.

\begin{proof}[Proof of Theorem \ref{thmWSP}.]
  \label{thmWSPproof}

  Let $F$ be a self-similar set in $\R^d$ not contained in any
  $(d-1)$-dimensional hyperplane and assume that the weak separation
  property is satisfied.   Also, without loss of generality, assume that $0 \in F$.  Let $c_{\min}=\min\{\, c_i : i\in\I\,\}$ and write $\lvert F \rvert$ for the diameter of $F$.
To prove the result, it suffices to show
  that $F$ is Ahlfors regular.  Let $s = \dimh F$ denote the Hausdorff
  dimension of $F$.  First observe that the weak separation property
  implies that $F$ is an $s$-set, i.e.
  \[
  0 < \mathcal{H}^s (F) < \infty,
  \]
  which is a necessary condition for Ahlfors regularity.  Indeed, any
  self-similar set has finite Hausdorff measure in the critical
  dimension and Zerner shows that, together with $F$ not being
  contained in a hyperplane, the weak separation property implies that
  this measure is also positive \cite[Corollary to Proposition
  2]{Zerner1996}.  Let $x \in F$ and $r \in (0, \min\{1, \lvert F \rvert\})$ be arbitrary.  The lower
  bound in the definition of Ahlfors regularity is easy to establish.  Choose $\alpha \in \I^*$ such that
  \[
  x \in S_\alpha(F) \words{and} c_\alpha \leq r / \lvert F \rvert <
  c_{\overline{\alpha}}.
  \]
  It is clear that such an $\alpha \in \I^*$ exists, that it satisfies
  $c_\alpha > c_{\min} r/\lvert F \rvert$ and that $S_\alpha(F)
  \subseteq B_r(x)$.  As such, using the scaling property for Hausdorff
  measure,
  \[
  \mathcal{H}^s \big(B_r(x) \cap F \big) \geq \mathcal{H}^s \big(
  S_\alpha(F) \big) \geq c_\alpha^s \, \mathcal{H}^s (F) \geq \ \frac{
    c_{\min}^s \, \mathcal{H}^s (F) }{\lvert F \rvert^s} \, r^s.
  \]
  The upper bound is more awkward to establish and relies on the weak
  separation property.
  Observe that
  \begin{plain}$$\eqalign{ \mathcal{H}^s \big(B_r(x) \cap F \big) &
      \leq \mathcal{H}^s \bigg( \bigcup_{\alpha \in \I_r(x) }
      S_\alpha(F) \bigg) \cr &\leq \ \sum_{\alpha \in \I_r(x) }
      c_\alpha^s \, \mathcal{H}^s (F) \cr &\leq \ \mathcal{H}^s (F) \,
      \lvert \I_r(x) \rvert \, r^s .  }$$\end{plain}%
  Thus in order to finish the proof, it suffices to bound $\lvert
  \I_r(x) \rvert$ independently of $x$ and $r$.  Of course, if the
  defining system has complicated overlaps, this is impossible, but we
  can prove it in our case by applying one of the equivalent
  formulations of the weak separation property given by Zerner.  In
  particular, \cite[Theorem 1 (3a), (5a)]{Zerner1996} shows the
  definition we use in this paper is equivalent to the following (stated in our notation): for all $y \in \R^d$, there exists
  $l(y) \in \mathbb{N}$ such that for any $\beta \in \I^*$ and $r>0$,
  every ball with radius $r$ contains at most $l(y)$ elements of the
  set
  \[
  \{ S_\alpha (S_\beta(y)) : \alpha \in \I_r \}.
  \]
  Set $y=0$ and choose $\beta \in \I$ arbitrarily and observe that if $\alpha \in \I_r(x)$, then
\[
S_\alpha(S_\beta(0)) \in S_\alpha(F) \subseteq B(x, r+r\lvert F \rvert).
\]
Since $\mathbb{R}^d$ is a doubling metric space we can find a cover of $ B(x, r+r\lvert F \rvert)$ by fewer than $L$ balls of radius $r$, where $L$ is a uniform constant independent of $x$ and $r$.  Each of these $r$-balls can contain no more than $l(0)$ of the points $\{S_\alpha(S_\beta(0))  :  \alpha \in \I_r \}$ and so we deduce that $\lvert  \I_r(x) \rvert \leq l(0) \, L$, completing the proof.
\end{proof}

\section{Systems without the Weak Separation Property}\label{WSPfalse}

Our second result shows that if the weak separation property does not
hold, then the Assouad dimension is bounded below by one.  On the real
line this gives a set of maximal dimension.

\begin{theorem}\label{thmnotWSP}
  Let $F$ be a non-trivial self-similar set in $\R$.  If the defining IFS for $F$ does not satisfy the weak
  separation property, then $ \dima F = 1.$
\end{theorem}

In $\R^d$ we can only show that the dimension is at least one.

\begin{theorem}\label{thmnotWSPRd}
  Let $F$ be a self-similar set in $\R^d$ not contained in
  any $(d-1)$-dimensional hyperplane.  If the defining IFS for $F$ does not satisfy the weak
  separation property, then $ \dima F \ge 1.$
\end{theorem}

Theorem~\ref{thmnotWSP} is the special case of
Theorem~\ref{thmnotWSPRd} when $d=1$, however, for clarity of
exposition we give a separate proof of Theorem~\ref{thmnotWSP} first.  Another reason for doing this is that Theorem~\ref{mainR} is the main result of the paper as it provides a precise result in $\R$ and so it is expedient to give a clear proof of Theorem~\ref{thmnotWSP} without the extra technical details required in $\R^d$.
Note that combining Theorem~\ref{thmWSP} with Theorem~\ref{thmnotWSP}
yields Theorem~\ref{mainR} which provides a precise dichotomy for
self-similar sets in the line.  Combining Theorem~\ref{thmWSP} with
Theorem \ref{thmnotWSPRd} yields Theorem~\ref{mainRd}.

One of the most powerful techniques for proving lower bounds for
Assouad dimension is to construct weak tangents to the set.  This
approach has been pioneered by Mackay and Tyson \cite{Mackay2010}, but also harks back to Furstenberg's notion of \emph{microsets} \cite{furstenberg}.  We first recall the
definition of the Hausdorff distance and then give the definition of
weak tangent from \cite{Mackay2010}, see also \cite{Mackay2011}.

\begin{definition}
  Given compact subsets $X$ and $Y$ of $\R^d$, the Hausdorff distance
  between $X$ and $Y$ is defined as
$$
\dh(X,Y)=\max\big\{\ph(X,Y),\ph(Y,X)\big\}$$ where
$$
\ph(X,Y)=\sup_{x\in X} \inf_{y\in Y} \|x-y\|.
$$
\end{definition}

\begin{definition}\label{weaktang}
  Let $X$ be a compact subset of $\R^d$ and let $F$ and $\hat F$ be compact subsets of $X$.  We say that $\hat F$
  is a weak tangent of $F$ if there exists a sequence of similarity maps
  $T_k\colon\R^d\to\R^d$ such that $d_{\cal H}(T_k(F) \cap X,
  \hat{F})\to 0$ as $k\to\infty$.
\end{definition}

\begin{proposition}\label{lowerbound}
  If $\hat F$ is a weak tangent of $F$ then $\dima \hat{F} \leq \dima
  F$.
\end{proposition}

The proof of Proposition \ref{lowerbound} can be found in
\cite[Proposition 6.1.5]{Mackay2010} and \cite[Proposition
2.1]{Mackay2011}. In practise, the similarity maps $T_k$ in Definition \ref{weaktang} are usually taken to be expanding with expansion ratios tending to infinity with $k$.  The result is still true if the similarity maps are contracting with contraction ratios tending to zero, but then the limit set $\hat F$ is often referred to as an \emph{asymptotic cone}, rather than a tangent.  In our work we shall employ the following slightly weaker notion of a
tangent than the weak tangent of Definition~\ref{weaktang}.

\begin{definition}\label{pseudotang}
  Let $F$ and $\hat F$ be compact subsets of $\R^d$.  We say $\hat F$
  is a weak pseudo-tangent of $F$ if there exists a sequence of
  similarity maps $T_k\colon\R^d\to\R^d$ such that
	$\ph(\hat{F},T_k(F))\to 0$ as $k\to\infty$.
\end{definition}

We choose the name weak pseudo-tangent because the main difference
between Definition \ref{pseudotang} and Definition \ref{weaktang} is
the replacement of the Hausdorff metric $\dh$ with the one-sided
Hausdorff pseudo-metric $\ph$.  Note that we also do not need to take
the intersection with some explicitly chosen set $X$ in the definition
of the weak pseudo-tangent.  Fortunately for our analysis the Assouad
dimension of a pseudo-tangent still forms a lower bound on the Assouad
dimension of the original set.

\begin{proposition} \label{wpseudotangdim} If $\hat F$ is a weak
  pseudo-tangent of $F$ then $\dima \hat{F} \leq \dima F$.
\end{proposition}
The proof of Proposition \ref{wpseudotangdim} follows the proofs in
the weak-tangent setting given by Mackay and Tyson, but we include the
details for completeness.  In fact our proof is slightly simpler since
we do not have to take intersections with some previously chosen $X$
and we need only control `one side' of the
convergence.  Moreover, Proposition \ref{lowerbound} follows from
Proposition \ref{wpseudotangdim} by choosing an appropriate reference set $X$.

\begin{proof}
  Let $s>\dima F$.  By definition there exists $K_s$ such that
$$
\cover_F(r,\rho)\le K_s (r/\rho)^s \words{for all} 0<\rho<r\le 1.$$
Since $\hat F$ is a weak pseudo-tangent there exist similarities
$T_k$ with corresponding contraction ratios $u_k$  such that
$$
\ph(\hat{F},T_k(F))\to 0 \words{as} k\to\infty.
$$
Since any cover of an $r$-ball in $T_k(F)$ by $\rho$-balls gives rise to a cover of a corresponding $r/u_k$-ball in $F$ by the same number of $\rho/u_k$-balls, and vice-versa,
it follows that
$$
\cover_{T_k(F)}(r,\rho) =\cover_{F}(r/u_k,\rho/u_k) \le K_s (r/\rho)^s
\words{for all} 0<\rho<r\le 1.$$
Note that $K_s$ is independent of
$k$.  Given $r \in (0, 1/2)$ and $\rho \in (0,r)$ choose $k$ so large that
$$
\ph(\hat{F},T_k(F))\le \rho/4.
$$
Thus, every point of $\hat F$ is within $\rho/2$ distance from some
point of $T_k(F)$ and in particular
\begin{equation}\label{fattk}
  \hat F\subseteq \bigcup_{\eta\in T_k(F)} B_{\rho/2}(\eta).
\end{equation}

Given $x\in \hat F$ choose $y\in T_k(F)$ such that $\|x-y\|<\rho/2$.
Then $$B_r(x)\subseteq B_{r+\rho/2}(y)\subseteq B_{2r}(y).$$ Let
$\{\,B_{\rho/2}(y_i): i=1,\ldots,N\,\}$ be a cover of $T_k(F)\cap
B_{2r}(y)$ with
\[
N \leq K_s \bigg( \frac{2r}{\rho/2} \bigg)^s.
\]
It follows that
$$
\hat F\cap B_r(x) \ \subseteq \bigcup_{\eta\in T_k(F)\cap B_{2r}(y)}
B_{\rho/2}(\eta) \ \subseteq  \ \bigcup_{i=1}^N \bigcup_{\eta\in
  B_{\rho/2}(y_i)} B_{\rho/2}(\eta) \ = \  \bigcup_{i=1}^N B_{\rho}(y_i).
$$
Therefore
$$
\cover_{\hat F\cap B_r(x)}(\rho)\le N\le K_s 4^s (r/\rho)^s$$
for all $x\in \hat F$ and $0<\rho<r<1/2$ which is sufficient to prove $\dima \hat F \le s$.  It
follows that $\dima \hat F \le \dima F$.
\end{proof}

\subsection{Proof of Theorem~\ref{thmnotWSP}}

Similarities $S_i$ on the real line with ratio $c_i>0$ come in two
types: with reflection and without.  Namely,
$$
S_i(x)=-c_i x + b_i \words{and} S_i(x)=c_ix + b_i.
$$
In the first case the derivative $S_i'=-c_i<0$.  In order to treat
iterated function systems that contain both types of similarities we
first prove the following result.

\begin{lemma}\label{noreflect}
  Suppose that there exist $\alpha_k,\beta_k\in\I^*$ such that
$$\|S_{\alpha_k}^{-1}\circ S_{\beta_k}-I\|_\infty\to 0
\words{as} k\to\infty.$$ Then there exist $\alpha'_k,\beta'_k\in\I^*$
such that
$$\|S_{\alpha'_k}^{-1}\circ S_{\beta'_k}-I\|_\infty\to 0
\words{as} k\to\infty$$ and $S_{\alpha'_k}'=c_{\alpha'_k}>0$ for all
$k$.
\end{lemma}

\begin{proof}
  If there are an infinite number of $k$ such that $S_{\alpha_k}'>0$
  then let $\alpha'_k$ be this subsequence and we are done.
  Otherwise, there must be an infinite number of $k$ such that
  $S_{\alpha_k}'<0$.  By taking a subsequence we may assume, in fact,
  that $S_{\alpha_k}'<0$ for all $k$.

  Now, since $S_{\alpha_k}$ contains a reflection for every $k$, there
  must be at least one similarity in the iterated function system that
  contains a reflection.  Without loss of generality, assume that
  $S_1'=-c_1<0$.  Define $\alpha'_k$ and $\beta'_k$ so that
$$
S_{\alpha'_k}^{-1}= S_1^{-1}\circ S_{\alpha_k}^{-1} \words{and}
S_{\beta'_k}= S_{\beta_k}\circ S_1.
$$
It follows that $S_{\alpha'_k}'>0$ for all $k$ and furthermore that
\begin{plain}$$\eqalign{ |(S_{\alpha'_k}^{-1}\circ S_{\beta'_k}-I)(x)|
    &= |S_1^{-1}\big( S_{\alpha_k}^{-1}\circ S_{\beta_k}\circ
    S_1(x)\big)-S_1^{-1}\big(S_1(x)\big)|\cr &=
    c_1^{-1}|S_{\alpha_k}^{-1}\circ S_{\beta_k}\circ S_1(x)-S_1(x)|\cr
    &= c_1^{-1}|(S_{\alpha_k}^{-1}\circ S_{\beta_k} -I )(S_1(x))|.
  }$$\end{plain}%
Since $F\subseteq [0,1]$ it follows that $S_1(x)\in [0,1]$ for $x\in
[0,1]$.  Therefore
$$
\|S_{\alpha'_k}^{-1}\circ S_{\beta'_k}-I\|_{\infty} \le c_1^{-1}
\|S_{\alpha_k}^{-1}\circ S_{\beta_k}-I\|_{\infty} \to 0,
$$
which finishes the proof of the lemma.
\end{proof}

We are now ready to prove Theorem \ref{thmnotWSP}.

\begin{proof}[Proof of Theorem \ref{thmnotWSP}]
  Since $F$ is non-trivial there must be two distinct points $a,b\in
  F$ and similarities $S_1$ and $S_2$ such that
  \begin{equation}\label{star}
    S_1(a)=a \words{and} S_2(b)=b.
  \end{equation}
  Choose $\rho>0$ small enough that $B_{2\rho}(a)\cap
  B_{2\rho}(b)=\emptyset$.

  Since the Weak Separation Property does not hold, $I\in \overline {{\cal E}\setminus\{ I\}}$ and we can choose
  $\alpha_k,\beta_k\in \I^*$ such that
$$0<\|S_{\alpha_k}^{-1}\circ S_{\beta_k}-I\|_\infty\to 0
\words{as} k\to\infty.$$ By Lemma \ref{noreflect} we may assume that
$$S_{\alpha_k}'=c_{\alpha_k}>0\words{for all} k.$$
Define $\varphi_k=S_{\alpha_k}^{-1}\circ S_{\beta_k}-I$.
Note that
$\varphi_k\ne 0$ for all $k$.  Therefore either $\varphi_k$ is a
similarity or a non-zero constant function.
Since similarities are injective, images of the disjoint
$2\rho$ balls $B_{2\rho}(a)$ and $B_{2\rho}(b)$ are disjoint.  In
particular, the origin can lie in at most one of these images.
On the other-hand, if $\varphi_k$ is a non-zero constant function,
then the image of both of these balls consists of the same non-zero
point and in particular does not contain the origin.
It
follows that each $\varphi_k\colon\R\to\R$ must satisfy at least one
of the following four conditions:
$$	\varphi_k(B_{2\rho}(a))\subseteq (0,\infty),\qquad
\varphi_k(B_{2\rho}(b))\subseteq (0,\infty),
$$
$$	\varphi_k(B_{2\rho}(a))\subseteq (-\infty,0)\words{or}
\varphi_k(B_{2\rho}(b))\subseteq (-\infty,0).
$$
Moreover, by taking a subsequence we may assume that all $\varphi_k$
satisfy the same inclusion.  For definiteness assume
$\varphi_k(B_{2\rho}(a))\subseteq (0,\infty)$ for all $k$.  The other
cases may be obtained from the first by relabeling $a$ and $b$ and
making a change of coordinates to reverse the direction of the real
line if necessary.

Let $f=S_1 \circ S_1$.  Then $f'=c=c_1^2>0$.  This avoids the changes
in sign that might occur when $S_1$ contains a reflection.
Fix $M$ large enough that $f^M(F)\subseteq B_\rho(a)$
(see (\ref{star})).  Define
\begin{plain}$$\eqalign{ \zeta_k&=\inf\big\{\, \varphi_k(x) : x\in
    B_{2\rho}(a)\,\big\},\cr
	\delta_k&=\inf\big\{\, \varphi_k(x) :
    x\in B_\rho(a)\,\big\}, \cr \text{and} \qquad
    \Delta_k&=\sup\big\{\, \varphi_k(x) : x\in B_\rho(a)\,\big\}.\cr
  }$$\end{plain}%
Then
$$
\delta_k-\zeta_k=\rho|\varphi_k'| \words{and}
\Delta_k-\delta_k=2\rho|\varphi_k'|.
$$
Therefore (since $\zeta_k\ge0$)
$$
\rho|\varphi_k'|\le\delta_k \words{and} \Delta_k\le 3\delta_k.
$$
While the inequality $\Delta_k\le 3\delta_k$ holds whether
$\varphi_k$ is a similarity
or a non-zero constant function,
note that in the latter case
$\varphi_k'=0$ and then $\Delta_k=\delta_k$.
In either case,
since we know $\|\varphi_k\|_\infty\to 0$ then $\delta_k\to 0$ as
$k\to\infty$.

We now derive an estimate that will be used as the basis for our
inductive construction in the next paragraph.
First note that
\begin{plain}$$\eqalign{ (f^{-m}\circ S_{\alpha_k}^{-1}&\circ
    S_{\beta_k}\circ f^{m}-I) (x)\cr &=
    f^{-m}\big(S_{\alpha_k}^{-1}\circ S_{\beta_k}\circ f^{m}(x)\big)
    -f^{-m}\big(f^m(x)\big)\cr &= c^{-m}\big(S_{\alpha_k}^{-1}\circ
    S_{\beta_k}\circ f^{m}(x) -f^m(x)\big)\cr &= c^{-m} \varphi_k
    \big(f^{m}(x)\big).  }$$\end{plain} Since $f^{m}(x)\in B_\rho(a)$
for every $x\in F$ and $m\ge M$ we obtain
\begin{equation}\label{gammaest}
  c^{-m}\delta_k\le
  (f^{-m}\circ S_{\alpha_k}^{-1}\circ S_{\beta_k}\circ f^{m}-I) (x)
  \le 3c^{-m}\delta_k
\end{equation}
for every $x\in F$ and $m\ge M$.

Given $\epsilon>0$
we now define natural numbers $k_j$, $m_j$, and maps $g_j$ and $h_j$
by induction on $j$.
First, choose $k_1$ and $m_1\ge M$
so large that
$$
c^{-m_1}\delta_{k_1}<\epsilon\le c^{-m_1-1}\delta_{k_1}$$
(this can be done by first choosing $k_1$ large enough that
$\delta_{k_1}<\epsilon c^{M}$ and then choosing $m_1$
appropriately) and define
$$g_1=f^{-m_1}\circ S_{\alpha_{k_1}}^{-1}
\words{and} h_1=S_{\beta_{k_1}}\circ f^{m_1}.
$$
From the estimate \eqref{gammaest} and our choice of $k_1$ and $m_1$
it follows that
$$
c \epsilon\le (g_1\circ h_1-I)(x)\le 3\epsilon
$$
holds for every $x\in F$.  Let $d_j=g_j'$ be the derivative of $g_j$.
Since both $f$ and $S_{\alpha_{k_1}}$ have positive derivatives, we
know that $d_1>0$.  Thus $d_1$ is the ratio corresponding
to the similarity $g_1$ such that $|g_1(x)-g_1(y)|=d_1|x-y|$
for every $x,y\in\R$.
By the construction below, it will also follow
that $d_j>0$ for every $j\in\N$.

For $j\ge 2$ choose $k_j$ and $m_j\ge M$ so large that
$$
d_{j-1} c^{-m_j}\delta_{k_j}<\epsilon\le d_{j-1}
c^{-m_j-1}\delta_{k_j}
$$
(this can be done in a similar way to the choice of $k_1$ and $m_1$,
above) and define
$$
g_j=g_{j-1}\circ f^{-m_j}\circ S_{\alpha_{k_j}}^{-1} \words{and}
h_j=S_{\beta_{k_j}}\circ f^{m_j}\circ h_{j-1}.
$$
Since $d_j=d_{j-1} c^{-m_j} c_{\alpha_{k_j}}^{-1}$ and $d_{1}>0$,
it follows by induction that $d_j>0$.  Since $h_{j-1}(x) \in F$ we
have $f^{m_j} (h_{j-1}(x)) \in B_\rho(a)$ and so (recalling that
$\varphi_k=S_{\alpha_k}^{-1}\circ S_{\beta_k}-I$)
\begin{plain}$$\eqalign{ (g_j\circ h_j- g_{j-1}\circ h_{j-1})(x)
    &=g_{j-1}\circ f^{-m_j}(S_{\alpha_{k_j}}^{-1} \circ
    S_{\beta_{k_j}}\circ f^{m_j}\circ h_{j-1}(x))\cr &\qquad-
    g_{j-1}\circ f^{-m_j}(f^{m_j}\circ h_{j-1}(x))\cr &=d_{j-1}
    c^{-m_j} \varphi_{k_j} (f^{m_j}\circ h_{j-1}(x))
  }$$\end{plain}%
implies for every $x\in F$ that
$$
c\epsilon\le (g_j\circ h_j- g_{j-1}\circ h_{j-1})(x)\le 3\epsilon.
$$

We next claim for all $n\in\N$ that
\begin{equation}\label{tanpts}
\{ a\}\cup \{\, g_j\circ h_j(a) : j=1,\ldots,n\,\}
\subseteq \{\, g_n\circ S_\beta(a) :\beta\in\I^*\,\}.
\end{equation}
This follows
from induction on $n$.  For $n=1$ we readily see this is true by
taking $\beta\in \I^*$ so that $S_\beta=h_1$ and then taking $\beta$
so that $S_\beta=g_1^{-1}$.  Suppose the claim holds true for $n-1$,
then
$$\{a\}\cup\{\, g_j\circ h_j(a) : j=1,\ldots,n-1\,\}
\subseteq \{\, g_{n-1}\circ S_\beta(a) :\beta\in\I^*\,\}.$$ Choose
$\gamma\in\I^*$ so that $S_\gamma=S_{\alpha_{k_n}}\circ f^{m_n}$.
Since $g_n=g_{n-1}\circ f^{-m_{n}}\circ S_{\alpha_{k_n}}^{-1}$, it
follows that
\begin{plain}$$\eqalign{ \{\, g_n\circ S_\beta(a) :\beta\in\I^*\,\}
    &\supseteq \{\, g_n\circ S_\gamma\circ S_\beta(a)
    :\beta\in\I^*\,\}\cr &= \{\, g_{n-1}\circ S_\beta(a)
    :\beta\in\I^*\,\}\cr &\supseteq \{a\}\cup\{\, g_j\circ h_j(a) :
    j=1,\ldots,n-1\,\}.  }$$\end{plain}%
Additionally taking $\beta\in\I^*$ so that $S_\beta=h_n$ completes the
induction.

We finish the proof by constructing a weak pseudo-tangent to $F$ that
has Assouad dimension $1$.  Let $g_{j,n}$ be the functions $g_j$
defined above when $\epsilon=(c n)^{-1}$ and define $T_n=g_{n,n}$.
It follows from \eqref{tanpts} that
\begin{plain}$$\eqalign{ T_nF&\supseteq \{\, g_{n,n}\circ S_\beta (a)
    : \beta\in \I^*\,\}\cr &\supseteq \{a\}\cup \{\, g_{j,n}\circ
    h_{j,n}(a) : j=1,\ldots,n\,\}.  }$$\end{plain}%
Since $ n^{-1}\le g_{j,n}\circ h_{j,n}(a)- g_{j-1,n}\circ h_{j-1,n}(a)
\le 3(c n)^{-1} $ we obtain
$$ \ph([a,a+1],T_n(F))\le 3 (c n)^{-1}\to 0.$$
Therefore $[a,a+1]$ is a weak pseudo-tangent of $F$ and
$$
\dima F\ge \dima [a,a+1] = 1.
$$
This finishes the proof of our lower bound for self-similar sets on
the real line that do not satisfy the weak separation condition.
\end{proof}

\subsection{Proof of Theorem~\ref{thmnotWSPRd}}

In this section we prove the more complicated case
concerning self-similar sets $F$ in $\R^d$.  Thus, $F$ is the attractor
of an iterated function system $\{S_i\}_{i\in\I}$ given by
$$
	S_i(x)=c_i O_i x + b_i
$$
where $O_i\in O(d)$ is a $d\times d$ orthogonal matrix, $c_i\in(0,1)$ and
$b_i\in\R^d$ for each $i\in\I$.
In the case of the real line
$O_i=\pm1$, and so $O_i^2$ always yielded the identity.
However, in higher dimensions there are orthogonal matrices such that
$O^m\neq I$ for all $m\ge0$.  To deal with this
we make use of Lemma \ref{lemG} below.
Given any $\gamma\in\I^*$, define $O_\gamma\in
O(d)$ and $b_\gamma\in\R^d$ such that
$$
S_\gamma(x)=c_\gamma O_\gamma x+b_\gamma.
$$
Note that $S_\gamma$ is a similarity with
ratio $c_\gamma$ such that
$$
	\|S_\gamma(x)-S_\gamma(y)\|=c_\gamma \|x-y\|
	\words{for all} x,y\in\R^d.$$
We now define
$$
	{\cal G}=\{\,O_\alpha : \alpha\in\I^*\,\}
\words{and}
	G=\overline {\cal G}.
$$
and show that for any $\epsilon>0$ there exists a finite collection of elements
of the form $O_\alpha$ with $\alpha\in I^*$ which can be used to approximate every element of $G$ to within $\epsilon$.

\begin{lemma}\label{lemG}
The set $G$ is a compact group.  Consequently,
given $\epsilon>0$ there exists a finite set $\J\subseteq\I^*$
such that for every $U\in G$ there is $\alpha\in\J$
such that $\|U-O_\alpha\|<\epsilon$.
\end{lemma}
\begin{proof}
Let $O(d)$ be the group of all $d\times d$ orthogonal matrices.
We claim that $G$ is a subgroup of $O(d)$.  Indeed, since $G\subseteq O(d)$
and ${\cal G}$ is a semigroup, then $G$ is at least a semigroup.
Let $U\in G$ and consider the sequence of iterates $U^n$ where
$n\in\N$.  Since $O(d)$ is compact and $G$ is closed, there exists
a subsequence $U^{n_k}\to V$ where $V\in G$.  It follows that
$U^{n_{k+1}-n_k}\to I$ and consequently $I\in G$.
We may further assume that $n_{k+1}-n_k\ge 2$ for every $k$, from which it
follows that $U^{n_{k+1}-n_k-1}\to U^{-1}$
and consequently $U^{-1}\in G$.  This proves the claim.

Now, since $G$ is compact,
$$
	G\subseteq \bigcup_{\alpha\in\I^*} B_\epsilon(O_\alpha)
\words{implies that}
	G\subseteq \bigcup_{\alpha\in\J} B_\epsilon(O_\alpha)
$$
for some finite subset $\J\subseteq\I^*$.
Let $U\in G$.  Then $U\in B_{\epsilon}(O_\alpha)$
for some $\alpha\in\J$ and consequently $\|U-O_\alpha\|<\epsilon$.
\end{proof}

For self-similar sets in $\R^d$ we insist that $F$ not lie in any
$(d-1)$-dimensional hyperplane of $\R^d$.
To make use of this hypothesis we first show that we
can find a collection of fixed points of the similarities which span $\R^d$, after a suitable translation.

\begin{lemma}\label{fixpoints}
Suppose that $F$ is a self-similar set in $\R^d$
that is not contained in any hyperplane.
Then there exist $\gamma_n\in\I^*$, where $n=1,\ldots, d+1$,
such that the fixed points $a_n$ of the maps $S_{\gamma_n}$
satisfy $\span\{\,a_n-a_{d+1} : n=1,\ldots, d\,\}=\R^d$.
\end{lemma}
\begin{proof}
Given $a_n\in\R^d$ let $A$ be the $d\times d$ matrix
given by the vectors $a_n-a_{d+1}$ as
$$
	A=\Big[a_1-a_{d+1}\Big| a_2-a_{d+1} \Big|\cdots \Big|a_d-a_{d+1}\Big].
$$
We are looking for $a_n$ such that $\rank A = d$ or
equivalently such that $\det A\ne 0$.
Since $\det A$ depends continuously on the $a_n$ then
$$
	{\cal R}=\{\, (a_1,\ldots,a_{d+1}) : \det A\ne 0\,\}
		\subseteq [\R^d]^{d+1}
$$
is open.  Since $F$ is not contained in any hyperplane,
there exist $b_n\in F$, $n=1,\ldots,d+1$, such that $(b_1,\ldots,b_{d+1})\in {\cal R}$.
Moreover, since ${\cal R}$ is open, there is $\epsilon>0$ such
that $|a_n-b_n|<\epsilon$ for $n=1,\ldots, d+1$ implies that
$(a_1,\ldots,a_{d+1})\in{\cal R}$.

Choose $r<c_{\rm min}$ such that $r \diam F< \epsilon$.  Then
$$b_n\in F=\bigcup_{\alpha\in \I_r} S_\alpha(F)
\words{implies that}
b_n\in S_{\gamma_n}(F)$$
for some $\gamma_n\in I_r$ (see (\ref{Iris}) for the definition of $I_r$).  Let $a_n$ be the fixed point
of $S_{\gamma_n}$.  Then
$$a_n\in F
\words{implies that}
a_n=S_{\gamma_n}(a_n)\in S_{\gamma_n}(F).$$
Therefore
$$|a_n-b_n|\le \diam S_{\gamma_n} F=c_{\gamma_n}\diam F
	\le r \diam F <\epsilon$$
and consequently  $(a_1,\ldots,a_{d+1})\in{\cal R}$.
\end{proof}

As the geometry in $\R^d$ is more complicated than on the real line
we will need to prove one more lemma before commencing with the proof
of Theorem \ref{thmnotWSPRd} itself.

\begin{lemma}\label{findab}
Let $F$ be a self-similar set in $\R^d$ not contained in
any $(d-1)$-dimensional hyperplane.
Let $\Phi_k$ be a sequence of affine linear
maps on $\R^d$ such that $0<\|\Phi_k\|_\infty\to0$
as $k\to\infty$.
Then there exist $\rho>0$, $\gamma\in\I^*$ and $a\in\R^d$
such that $S_\gamma(a)=a$ and a subsequence $k_j$ such that
for every $j$ we have
$$
\sup\{\, \|\Phi_{k_j}(x)\| : x\in B_\rho(a)\,\}
\le 3\inf\{\, \|\Phi_{k_j}(x)\| : x\in B_\rho(a)\,\}
$$
and
$\rho\|D\Phi_{k_j}\|\le \|\Phi_{k_j}(a)\|$.
\end{lemma}
\begin{proof}
Note that the derivative $D \Phi_k$ is a
$d\times d$ matrix.
We separate our proof into two cases:
the first case is when there are an infinite
number of $k$ such that the derivative $D\Phi_k=0$
and we may assume by taking subsequences that
$D\Phi_{k_j}=0$ for all $k_j$; the second case is when there are
only a finite number of $k$ such that $D\Phi_k=0$ and we may
assume $D\Phi_{k_j}\ne 0$ for all $j$.

If $D\Phi_{k_j}=0$ for all $j$, then $\Phi_{k_j}=B_j$ where $B_j\in\R^d$
and $B_j\ne 0$.  In this case any $\rho>0$ and $\gamma\in\I^*$ yields
a similarity $S_\gamma$ with fixed point $a$
such that $0\not\in\Phi_{k_j}(B_{2\rho}(a))$.

If $D\Phi_{k_j}\ne 0$ for all $j$, write
the singular value decomposition
of $D\Phi_{k_j}$ as
$$
	D\Phi_{k_j} = U_j \Sigma_j V_j^{-1}
$$
where $U_j,V_j\in O(d)$ are $d\times d$ orthogonal matrices
and $\Sigma_j$ is a diagonal matrix with singular
values ordered such that $\sigma_{j,1}\ge\cdots\ge \sigma_{j,d}\ge 0$.
Since $O(d)$ is compact, we may assume by taking further
subsequences, that $V_j\to V$ for some $V\in O(d)$.
Define
$$s_j=\sigma_{j,1}=\|D\Phi_{k_j}\|,
\quad u_j=U_j e_1,\quad
v_j=V_j e_1
\words{and}
v_*=V e_1,
$$
where $e_1 = (1,0, \dots, 0)$.  Then $0<\|D\Phi_{k_j}\|\to 0$ implies that
$0<s_j\to 0$ as $j\to\infty$.

Since $F$ is not contained in any hyperplane,
Lemma~\ref{fixpoints} implies that there exist $\gamma_n \in \I^*$
and $a_n\in F$ such that
\begin{equation}\label{nohyper}
	S_{\gamma_n}(a_n)=a_n
\words{and}
\span\{\, a_n-a_{d+1} : n=1,\ldots d\,\}=\R^d.
\end{equation}
By \eqref{nohyper} there exists $n$ such that
$$
(a_n-a_{d+1})\cdot v_*\ne 0.
$$
Define
	$\varphi_j(x)=u_j\cdot \Phi_{k_j}(x)$.
Note that
$$
	\varphi_j(x)=
	u_j\cdot (D\Phi_{k_j})x+z_j
    =U_j e_1\cdot U_j\Sigma_j V_j^{-1}x+z_j
    =s_j v_j\cdot x+z_j
$$
for the scalars $z_j=u_j\cdot \Phi_{k_j}(0)$.
Therefore
$$
	\varphi_j(a_n)-\varphi_j(a_{d+1})
    =s_j v_j\cdot (a_n-a_{d+1}).
$$
Since $v_j\to v_*$ and $s_j>0$ it follows for
$j$ large enough that
$\varphi_j(a_n)-\varphi_j(a_{d+1})\ne 0$.
By choosing $\rho>0$ small enough and taking
another subsequence, we may conclude that
$$\varphi_j(B_{2\rho}(a_{n}))\cap \varphi_j(B_{2\rho}(a_{d+1}))=\emptyset
\words{for every} j.$$
For each $j$ at least one of the statements $0\not\in\varphi_j(B_{2\rho}(a_{n}))$
or $0\not\in\varphi_j(B_{2\rho}(a_{d+1}))$ is true.
By taking yet one more subsequence we may
fix $a$ to be either $a_n$ or $a_{d+1}$ and
$\gamma$ to be $\gamma_n$ or $\gamma_{d+1}$, respectively,
so that
$0\not\in \varphi_j(B_{2\rho}(a))$ for every $j$.
It follows that $S_\gamma(a)=a$ and
$0\not\in \Phi_{k_j}(B_{2\rho}(a))$ for every $j$.

Define
\begin{plain}\begin{equation}\label{deltadef}\eqalign{
	\delta_k &=\inf\big\{\,\|\Phi_k(x)\|: x\in B_{\rho}(a)\,\big\}\cr
	\Delta_k &=\sup\big\{\,\|\Phi_k(x)\|: x\in B_{\rho}(a)\,\big\}.
  }\end{equation}\end{plain}%
If $D\Phi_k=0$ then $\delta_k=\Delta_k$.
Otherwise $0\not\in\varphi_k(B_{2\rho}(a))$ implies that
for $x\in B_{\rho}(a)$
$$
	\|\Phi_k(x)\|\ge |u_k\cdot\Phi_k(x)|
		= |s_kv_k\cdot x+z_k|\ge s_k\rho,
$$
and therefore $\delta_k\ge s_k\rho$.
On the other hand $\|\Phi_k(x)-\Phi_k(y)\|\le s_k \|x-y\|$
implies that
$$\Delta_k\le \delta_k + 2s_k\rho \le 3\delta_k.$$
This completes the proof of the lemma.
\end{proof}

We are now ready begin the proof of Theorem \ref{thmnotWSPRd}.
The proof will construct a weak tangent to $F$ that has
dimension greater than one; we recover a `one-dimensional' structure by constructing  a set all of whose points lie in a particular cone in $\R^n$. For a set $U\in\R^d$ we use $$
\cone(U)=\{\, \lambda u : \lambda>0\hbox{ and } u\in U\,\}
$$
for the cone generated by $U$.

\begin{proof}[Proof of Theorem \ref{thmnotWSPRd}]
Since the weak separation condition is not satisfied,
there exist $\alpha_k,\beta_k\in\I^*$ such that
$0<\|S_{\alpha_k}^{-1}\circ S_{\beta_k}-I\|_\infty\to 0$.
Define $\Phi_k=S_{\alpha_k}^{-1}\circ S_{\beta_k}-I$.
After taking subsequences we may assume by Lemma~\ref{findab}
that there exist
$\gamma\in\I^*$ and $a\in\R^d$ such that $S_\gamma(a)=a$
and $\rho>0$ such that the inequalities
\begin{equation}\label{D3d}
\sup\{\, \|\Phi_{k}(x)\| : x\in B_\rho(a)\,\}
\le 3\inf\{\, \|\Phi_{k}(x)\| : x\in B_\rho(a)\,\}
\end{equation}
and
$\rho\|D\Phi_k\|\le \|\Phi_k(a)\|$ are satisfied
for all $k$.
Defining $\Delta_k$ and $\delta_k$ as in \eqref{deltadef} we
may further express \eqref{D3d} as $\Delta_k\le 3\delta_k$.

Given
  $\eta\in(0,1)$, cover the unit sphere $\partial B_1(0)$ by a
  finite number $M$ of balls $B_{\eta/2}(w_m)$ where $\|w_m\|=1$
  and $m=1,2,\ldots,M$.  It follows that
$$
\R^d=\bigcup_{m=1}^M \cone(B_{\eta/2}(w_m)).
$$
For each $k$ choose $m_k$
such that $\Phi_k(a)/\|\Phi_k(a)\|\in B_{\eta/2}(w_{m_k})$.
Upon taking a subsequence, we obtain a single $w\in\partial B_1(0)$ such that
$${\Phi_k(a)/\|\Phi_k(a)\|}\in
	B_{\eta/2}(w)\words{for all} k.$$

Define $f=S_\gamma=c_\gamma O_\gamma x+b_\gamma$ and note
that $Df=c_\gamma O_\gamma$.
Let $\rho'=\rho\eta/5$ and
choose $M$ so large that $m\ge M$ implies that
$f^m(x)\in B_{\rho'}(a)$ for every $x\in F$.
Let $\epsilon_2=\eta/4$.
By Lemma \ref{lemG} there exists a finite set $\J\subseteq\I^*$
such that $U\in G$ implies there is $\alpha\in\J$
such that $\|U-O_\alpha\|<\epsilon_2$.

Given $\epsilon>0$
we now define natural numbers $k_j$, $m_j$
and maps $g_j$ and $h_j$ by induction on $j$.
Let $g_0=I$, $h_0=I$, $d_0=1$ and $W_0=I$.  Note that $Dg_0=d_0W_0$.
Choose $k_j$ and $m_j\ge M$ so large that
$$
	d_{j-1} c_*^{-1} c_\gamma^{-m_j}\delta_{k_j}<\epsilon
	\le d_{j-1} c_*^{-1} c_\gamma^{-m_j-1} \delta_{k_j}
$$
and let $\gamma_j\in \J$ be chosen so
$
\| O_{\gamma_j}^{-1} O_\gamma^{-m_j} W_{j-1}-I\|
=\| W_{j-1} O_{\gamma_j}^{-1} O_\gamma^{-m_j}-I\|
\le\epsilon_2$.  Define
$$
g_j=g_{j-1}\circ S_{\gamma_j}^{-1}\circ f^{-m_j}\circ S_{\alpha_{k_j}}^{-1}
\words{and}
h_j=S_{\beta_{k_j}}\circ f^{m_j}\circ S_{\gamma_j}\circ h_{j-1}.
$$
Let $d_j$ be the ratio for $g_j$ such that
$$
	\|g_j(x)-g_j(y)\|=d_j\|x-y\|\words{for all}x,y\in\R^d
$$
and let $W_j$ be the $d\times d$
orthogonal matrix such that $Dg_j=d_jW_j$.

We now show that the sequence of functions $g_j$ and $h_j$ defined
above satisfy
\begin{equation}
c_* c_\gamma\epsilon\le
	\|(g_j\circ h_j- g_{j-1}\circ h_{j-1})(x)\|
	\le 3\epsilon
\end{equation}
and
\begin{equation}
(g_j\circ h_j- g_{j-1}\circ h_{j-1})(x)\in\cone(B_\eta(w))
	\words{for all} x\in F
\end{equation}
for every $j\in\N$.
Since
$$
	f^{m_j}\circ S_{\gamma_j}\circ h_{j-1}(x)\in B_{\rho'}(a)
		\subseteq B_{\rho}(a)
\words{for all} x\in F,
$$
then
\begin{plain}$$\eqalign{ (g_j\circ h_j- g_{j-1}\circ h_{j-1})(x)
    &= g_{j-1}\circ S_{\gamma_j}^{-1}\circ
		f^{-m_j}(S_{\alpha_{k_j}}^{-1} \circ
    S_{\beta_{k_j}}\circ f^{m_j} \circ S_{\gamma_j}\circ h_{j-1}(x)) \cr
&\qquad-
g_{j-1}\circ
    S_{\gamma_j}^{-1}\circ
		f^{-m_j}(f^{m_j} \circ S_{\gamma_j}\circ h_{j-1}(x))\cr
	&=d_{j-1} c_{\gamma_j}^{-1}
    c^{-m_j} W_{j-1} O_{\gamma_j}^{-1}	
		O_\gamma^{-m_j}\Phi_{k_j} (f^{m_j}\circ S_{\gamma_j}\circ h_{j-1}(x))
  }$$\end{plain}%
which implies
$$
c_* c_\gamma\epsilon\le
	d_{j-1}c_{\gamma_j}^{-1} c^{-m_j} \delta_{k_j}\le
	\|(g_j\circ h_j- g_{j-1}\circ h_{j-1})(x)\|
	\le d_{j-1}c_{\gamma_j}^{-1} c^{-m_j} \Delta_{k_j} \le 3\epsilon
$$
for all $x\in F$.

Write $y=f^{m_j}\circ S_{\gamma_j}\circ h_{j-1}(x)$,
$z=\Phi_{k_j}(y)$ and
$\tilde z=W_j O_{\gamma_j}^{-1} O_\gamma^{-m_j} z$.
Then
$$
    \|z-\Phi_{k_j}(a)\|=
    \|\Phi_{k_j}(y)-\Phi_{k_j}(a)\|\le \|D\Phi_{k_j}\| \|y-a\|
    \le
\rho'
\|D\Phi_{k_j}\|
$$
which implies that
$$  \|\tilde z - z\|\le \epsilon_2 \|z\|
	\le \|z-\Phi_{k_j}(a)\|+\|\Phi_{k_j}(a)\|
        \le \epsilon_2 \rho'\|D\Phi_{k_j}\| +\epsilon_2 \|\Phi_{k_j}(a)\|.
$$
Since $\eta<1$, $\rho'=\rho\eta/5$, $\epsilon_2=\eta/4$
and $\rho \|D\Phi_{k_j}\|\le\|\Phi_{k_j}(a)\|$ it follows that
\begin{plain}$$\eqalign{
    \|\tilde z-\Phi_{k_j}(a)\|
        &\le (1+\eta/4)(\eta/5)\rho\|D\Phi_{k_j}\|+\eta/4 \|\Phi_{k_j}(a)\|\cr
        &\le \big((5/4)(\eta/5)+\eta/4\big) \|\Phi_{k_j}(a)\|
        = (\eta/2) \|\Phi_{k_j}(a)\|.
}$$\end{plain}%
Consequently
\begin{plain}\begin{equation}\label{incone}
{W_j O_{\gamma_j}^{-1} O_\gamma^{-m_j}
    \Phi_{k_j} \big(f^{m_j}\circ S_{\gamma_j}\circ h_{j-1}(x)\big)\over
		\|\Phi_{k_j}(a)\|}
    \in B_{\eta/2}\bigg({\Phi_{k_j}(a)\over \|\Phi_{k_j}(a)\|}\bigg)
        \subseteq B_{\eta}(w).
\end{equation}\end{plain}%
It follows that
$$
(g_j\circ h_j- g_{j-1}\circ h_{j-1})(x)\in\cone(B_\eta(w))
	\words{for all} x\in F.
$$
Figure 1 shows representative locations
for $(g_1\circ h_1- I)(a)$ and
$(g_2\circ h_2- g_1\circ h_1)(a)$ in the cone $\cone(B_\eta(w))$.
Note that the norm of the projection of the points $g_j\circ h_j(a)$
along the $w$ direction always increases as $j$ increases.
In particular, we have
$$
	c_* c_\gamma \xi \epsilon \le
	w\cdot (g_j\circ h_j-g_{j-1}\circ h_{j-1})(a)\le 3\epsilon
$$
where $\xi=\sqrt{1-\eta^2}$.

\begin{figure}[H]
  \centering
  \includegraphics[width=60mm]{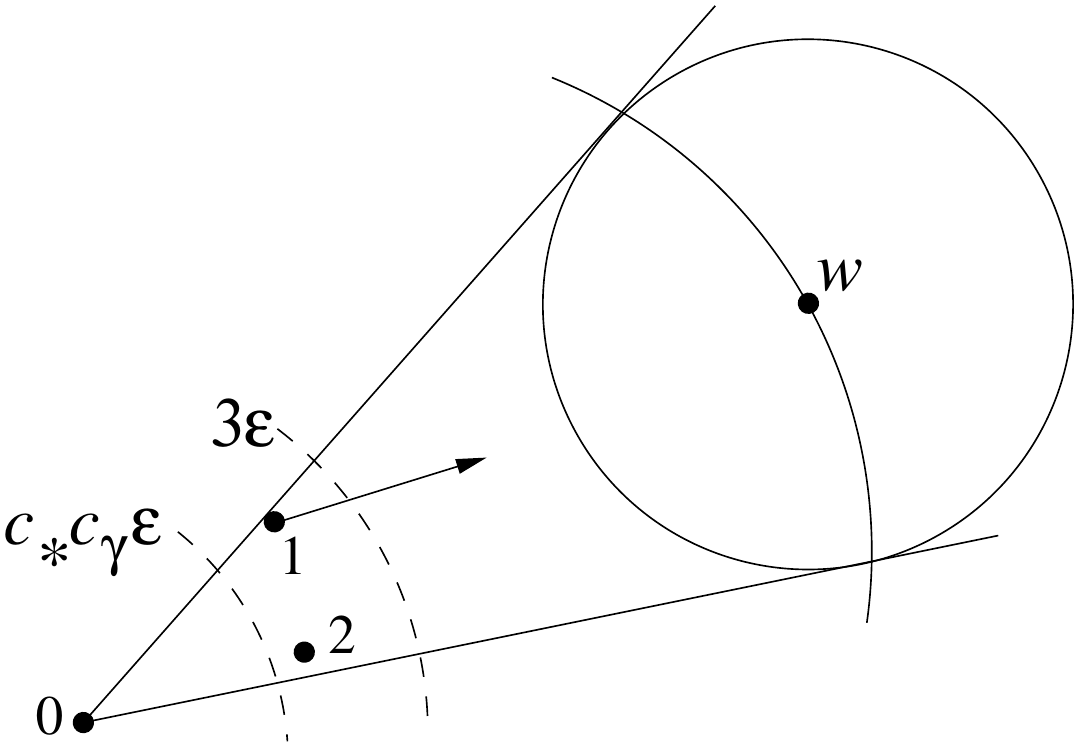}
	\caption{
        Locations of
		$(g_1\circ h_1- I)(a)$ and
		$(g_2\circ h_2- g_1\circ h_1)(a)$
		denoted by points $1$ and $2$.
        The arrow points to the location of $(g_2\circ h_2-I)(a)$.
    }
\end{figure}

We finish by constructing a weak tangent to $F$ that has
Assouad dimension 1.  The inclusion
$$
\{ a\}\cup \{\, g_j\circ h_j(a) : j=1,\ldots,n\,\}
\subseteq \{\, g_n\circ S_\beta(a) :\beta\in\I^*\,\}
$$
holds just as \eqref{tanpts} held in the proof for the real line.
However, unlike our proof for the real line,
we will construct a weak tangent rather than a weak
pseudo-tangent.  This is done in order to
make use of the fact that the space of all compact subsets of a compact set in $\R^d$
is a compact metric space with respect to the Hausdorff metric.

Let $g_{j,n}$ be the functions $g_j$ defined above when $\epsilon
	=(c_* c_\gamma \xi n)^{-1}$ and define $T_n=g_{n,n}$.  Note that the
cone $\cone(B_\eta(w))$ in the construction does not depend on $\epsilon$.
As before, it follows that
$$T_n F\supseteq \{a\}\cup\{\, g_{j,n}\circ h_{j,n}(a):j=1,\ldots n\}.$$
Since
$$
	n^{-1}\le w\cdot (g_{j,n}\circ h_{j,n}-g_{j-1,n}\circ h_{j-1,n})(a)
		\le 3(c_*c_\gamma\xi n)^{-1}
$$
we obtain
\begin{equation}\label{itsnotempty}
	\ph([a_w,a_w+1],w\cdot T_n(F))\le
		 3(c_*c_\gamma\xi n)^{-1}\to 0,
\end{equation}
where $a_w=w\cdot a$.
However $w\cdot T_n$ is not a similarity, so this is not a weak
pseudo-tangent and we
shall make use of one more compactness argument to finish the proof.

Let $$X=\overline{(\{a\}+\cone(B_\eta(w)))\cap B_{1/\xi}(a)}$$
and consider the sets
$B_n=T_n(F)\cap X$.  Since the space of all compact subsets of $X$
is a compact metric space with respect to the Hausdorff metric,
there exists a subsequence $n_k$ and a compact set $\hat F\in X$
such that
$$
	\dh(\hat F,T_{n_k}(F)\cap X)\to 0\words{as}k\to\infty.
$$
Thus, $\hat F$ is a weak tangent to $F$.
Since \eqref{itsnotempty} still holds along any subsequence, we
conclude that
$\hat F$ is not empty and in particular
$w \cdot\hat F=[a_w,a_w+1]$.
Since
$$\cover_{\hat F}(r,\rho)\ge
\cover_{B_r(a)\cap \hat F}(\rho)\ge \cover_{[a_w,a_w+\xi r)}(\rho)
\ge (2^{-1}\xi)(r/\rho)$$
for all $0<\rho<r<1$, it follows immediately that $\dima(\hat F)\ge 1$.
\end{proof}

\section{Examples}\label{examples}

For IFSs satisfying the WSP our result is sharp, and for IFSs not satisfying the WSP our result is sharp if we restrict our attention to $\R$. In this section we consider IFSs in $\mathbb{R}^d$ which do not satisfy the WSP and show that different phenomena are possible.
We conclude by mentioning our hope for future work
that leads to a refinement of our understanding of the
Assouad dimension in $\R^d$ when the weak separation
property does not hold.

\subsection{Intermediate Assouad dimension in the plane} \label{intexamplesect}

In light of Theorems \ref{thmWSP} and \ref{thmnotWSP}, one might naively 
imagine that the Assouad dimension of a self-similar set is either
equal to the Hausdorff dimension, or maximal (equal to the ambient
spatial dimension).  Trivially, this is false by noting that a
self-similar set in $\R$ with Assouad dimension equal to 1 and
similarity dimension strictly less than 1, is also a self-similar
subset of $\R^d$ via the natural embedding and this does not alter the
dimensions.  However, this is not very enlightening so we give a more
convincing counter example, where the self-similar set is not contained in any
hyperplane.  The key is to make the weak separation property fail
`only in one direction'.  We achieve this by modifying an example of
Bandt and Graf \cite{Bandt1992}, which shows that the weak separation
property can fail in the line, even if all the contraction ratios are
equal.  Consider the following four similarity maps on $[0,1]^2$:
$$ S_1(x) = {x/5}, \qquad S_2(x) = {x/5}+({t/5},0), $$
$$S_3(x) = {x/5} + ({4/5},0), \qquad S_4(x) = {x/5} +(0,{4/5}) $$
where $t \in [0,4]$.  Let $F$ denote the attractor of this system and
note that it is not contained in any line and has similarity dimension
equal to $\log 4 / \log 5$, which is an upper bound for the Hausdorff
dimension.  Let $F_1$ denote the projection of $F$ onto the first
coordinate and $F_2$ denote the projection of $F$ onto the second
coordinate.  $F_1$ is the self-similar attractor of the iterated
function system consisting of the three maps
\[
S'_1(x) = x/5, \qquad S'_2(x) = x/5+t/5, \qquad S'_3(x) = x/5 + 4/5
\]
defined on $[0,1]$ and $F_2$ is the self-similar attractor of the
iterated function system consisting of the two maps
\[
S''_1(x) = x/5, \qquad S''_2(x) = x/5 + 4/5
\]
defined on $[0,1]$, which satisfies the open set condition and thus
has Hausdorff dimension and Assouad dimension equal to $\log 2 / \log
5$.  Following Bandt and Graf \cite[Section 2 (5)]{Bandt1992}, if $t$
is rational, then the iterated function system for $F_1$ satisfies the
weak separation property.  However, it is possible to construct
irrational $t$ for which the weak separation property is not
satisfied.  Assume for a moment that we have done this for some $t \in
[0,4]$.  It follows from Theorem \ref{thmnotWSP} that $\dima F_1 = 1$.
Note that $F$ is contained in $F_1 \times F_2$ and contains an
isometrically embedded copy of $F_1$.  Since the Assouad dimension is
subadditive on product sets and monotone, the Assouad dimension of $F$
is bounded above by the sum of the Assouad dimensions of $F_1$ and
$F_2$ and below by the Assouad dimension of $F_1$.  This gives
\[
\dimh F \leq \log 4 / \log 5 < 1 \leq \ \dima F \leq 1 + \log 2 / \log
5 < 2
\]
as required.  Moreover, it is easy to see that this example could be modified to produce self-similar sets with Assouad dimension arbitrarily close to $1$, showing that the lower bound in Theorem \ref{mainRd} is sharp.

We will now show how to choose $t$ such that the weak separation
property fails for the defining iterated function system for $F_1$.  Following Bandt and Graf \cite[Section 2 (5)]{Bandt1992}, let
\[
t = 4 \sum_{k=0}^{\infty} 5^{-2^k} = 0.9664102400 \dots
\]
and observe that elements $ (S'_{\textbf{\emph{i}}})^{-1} \circ
S'_{\textbf{\emph{j}}} \in \mathcal{E}$ where $\lvert
\textbf{\emph{i}} \rvert = \lvert \textbf{\emph{j}} \rvert = n$ are
precisely maps of the form
\[
x \mapsto x + \sum_{k=0}^{n-1} 5^k a_k
\]
for any sequence $a_k$ over $\{0, \pm 4, \pm t, \pm(4-c)\}$.  Let
$n=2^m+1$ for some large $m \in \mathbb{N}$ and choose the
coefficients $\{a_k\}_{k=0}^{n-1}$ as follows:
\[
a_k = \left \{ \begin{array}{cc}
    -4 &  \qquad \text{if }  k = 2^m-2^l \text{ for some } l=0, \dots, m\\
    t &  \text{if } k=n-1 \\
    0 & \text{otherwise}\\
  \end{array} \right.
\]
This renders the map above equal to
\[
x \mapsto x + 5^{2^m} t - 4 \sum_{k=0}^{m} 5^{2^m-2^k} = x + 4
\sum_{k=m+1}^{\infty} 5^{2^m-2^k}
\]
which approximates the identity map arbitrarily closely by making $m$
(and thus $n$) large.

\begin{figure}[H]
  \centering
  \includegraphics[width=120mm]{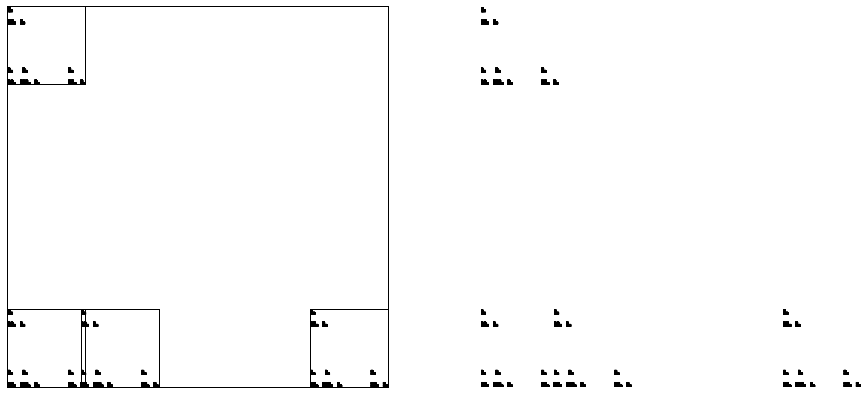}
  \caption{A self-similar set in the plane with intermediate Assouad
    dimension where $t = 0.9664102400 \dots$ is chosen as above to
    guarantee the failure of the weak separation property for $F_1$.
    The version on the left includes squares to indicate the images of
    the four maps in the iterated function system.}
\end{figure}

\subsection{Full Assouad dimension in $\R^d$}

Here we show how to generalise the example given in Fraser \cite[Section
3.1]{Fraser2013} to arbitrary dimensions.  In particular, we show how
to construct a self-similar set in $\R^d$ with arbitrarily small
Hausdorff dimension, but full Assouad dimension.

Let $d \in \mathbb{N}$ and $\Lambda = \big\{ z = (z_1, \dots, z_d) :
z_i \in \{0,1\} \text{ for $i = 1, \dots, d$}\big\}$ be the set of
$2^d$ corners of the unit hypercube $[0,1]^d$. Let $\alpha, \beta,
\gamma \in (0,1)$ be such that $(\log\beta)/(\log\alpha) \notin
\mathbb{Q}$ and define similarity maps $S_\alpha, S_\beta$ and
$S_\gamma^{z}$ ($z \in \Lambda \setminus \{(0, \dots, 0)\}$) mapping
$[0,1]^d$ into itself by
\[
S_\alpha(x) = \alpha x, \qquad S_\beta(x) = \beta x \qquad \text{and}
\qquad S_\gamma^{z}(x) = \gamma x +(1-\gamma)z .
\]
Let $F \subseteq [0,1]^d$ be the self-similar attractor of the
iterated function system consisting of these $1 + 2^d$ maps, which we
will now prove has full Assouad dimension, $\dima F = d$, by showing
that $[0,1]^d$ is a weak tangent to $F$.  Let $X = [0,1]^d$ and assume
without loss of generality that $\alpha<\beta$.  For each $k \in
\mathbb{N}$ let $T_k$ on $\R^d$ be defined by
\[
T_k(x) = \beta^{-k}x
\]
and we will now show that $T_k(F) \cap [0,1]^d \to_{d_\mathcal{H}}
[0,1]^d$.  Let
\[
E_k^d := \bigcup_{ z \in \Lambda} \Big( \big\{\alpha^m\beta^n z : m
\in \mathbb{N}, \, n \in \{-k, \dots, \infty\} \big\} \cap [0,1]^d
\Big)
\]
and observe that $E_k^d \subseteq T_k(F) \cap [0,1]^d$ for each $k \in
\mathbb{N}$.  However,
\[
E_k^d = \prod_{l=1}^d \Big( \big\{\alpha^m\beta^n : m \in \mathbb{N},
\, n \in \{-k, \dots, \infty\} \big\} \cap [0,1] \Big)
\]
where the set $\big\{\alpha^m\beta^n : m \in \mathbb{N}, \, n \in
\{-k, \dots, \infty\} \big\} \cap [0,1]$ is precisely the pre-tangent
set $E_k$ which appears in \cite[Section 3.1]{Fraser2013}.  There it
was shown that $E_k \to_{d_\mathcal{H}} [0,1]$ from which it follows
immediately that $E_k^d \to_{d_\mathcal{H}} [0,1]^d$ and thus $ T_k(F)
\cap [0,1]^d \to_{d_\mathcal{H}} [0,1]^d$, proving the desired result.
Proving that $E_k \to_{d_\mathcal{H}} [0,1]$ follows by applying
Dirchlet's Theorem on Diophantine approximation and using the fact
that $\alpha$ and $\beta$ are not log-commensurable. We omit the
details but refer the reader to \cite[Section 3.1]{Fraser2013}.

The similarity dimension of $F$, which is an upper bound for the
Hausdorff dimension, is the unique solution, $s$, of
\[
\alpha^s + \beta^s + (2^d-1)\gamma^s=1
\]
which, by making $\alpha, \beta$ and $\gamma$ small, can be made
arbitrarily small (but positive) whilst retaining the
non-log-commensurable condition on $\alpha$ and $\beta$.  Finally, observe that despite the
fact we were able to utilise the product \emph{sub}structure of $F$
and thus the product structure of the tangent, $F$ itself is not a
product, nor could we generalise the example in \cite[Section
3.1]{Fraser2013} by simply taking the $d$-fold product of the one
dimensional version because the result would not be a self-similar
set, but a strictly self-affine set.

\subsection{Future work}

The lower bound (of 1) in Theorem \ref{thmnotWSPRd} concerning systems which fail the WSP is sharp by virtue of Section \ref{intexamplesect}.  We suspect, however, that there is a
natural additional condition one could add to `failure of the weak
separation property' that would guarantee, for example, that $\dima
F\ge 2$.  Exactly what this additional condition might be is an open
question and a topic of further study.

\section*{Acknowledgements}

This work was carried out while J.M.F. was at the University of Warwick where he was supported by the EPSRC grant EP/J013560/1.
J.C.R.\ is supported by an EPSRC Leadership Fellowship, grant
EP/G007470/1; this grant also supported the visit of A.M.H.\ to
Warwick and partially supported the time E.J.O.\ spent at Warwick while
on sabbatical leave from the University of Nevada Reno.

\vspace{5mm}

\noindent \emph{Email addresses:}\\ \\
jon.fraser32@gmail.com\\
henderson@math.ucr.edu\\
ejolson@unr.edu\\
j.c.robinson@warwick.ac.uk

\end{document}